\documentclass[a4paper,12pt]{article}
\usepackage{amsmath}
\usepackage{amssymb}
\usepackage{amsthm}
\usepackage{amsfonts}
\usepackage{mathscinet}
\usepackage{enumerate}
\usepackage{pdfsync}
\usepackage{setspace}
\usepackage{mathrsfs}
\usepackage{stmaryrd}
\usepackage{tikz}
\usepackage{cite}

\usepackage{bookmark}
\usepackage{authblk}

\newcommand {\new} {\newcommand}
\newcommand \oper[2] {\new #1 {\operatorname{#2}}}

\newcommand \gcode[1] {\ulcorner\! #1 \!\urcorner}	
\newcommand \se [1] { \{ #1 \}}			
\newcommand\set [2]{ \{#1:#2\} }			
\newcommand\res {\!\upharpoonright\!}		

\newcommand\eqiv {\leftrightarrow}			
\newcommand\corner[1] {
  \langle #1 \rangle}		

\newcommand\coll{{\text{Coll}}}

\newcommand\concat {{^\frown}}   

\oper{\Ord}{Ord}				
\oper{\hull}{Hull}				
\oper{\ppt}{ppt} 				
\newcommand{\ord} {\Ord}
\oper{\ZFC}{ZFC}				
\oper{\rank}{rank}				
\oper{\crit}{cr}					
\oper{\crt}{crt}					
\oper{\cf}{cf}					

\oper{\height}{ht}				
\oper{\wfcore}{wfcore}					
\oper{\core}{core}					

\oper{\Ult}{Ult}				
\oper{\Cone}{Cone}				
\oper{\dirlim}{dirlim}
\oper{\rud}{rud}		
\oper{\const}{const}
\oper{\OD}{OD}
\oper{\final}{final}
\oper{\HYP}{HYP}
\oper{\wfp}{wfp}
\oper{\Hom}{Hom}
\new{\ult} {\Ult}
\oper{\dom}{dom}
\oper{\rep}{rep}

\oper{\suc}{succ}
\oper{\fac}{fac}

\oper{\Code}{Code}

\oper{\ran}{ran}
\oper{\maxdom}{maxdom}
\oper{\maxran}{maxran}

\oper{\lp}{Lp} 

\newtheorem{theorem}{Theorem}[section]

\newtheorem{lemma}[theorem]{Lemma}

\newtheorem{claim}[theorem]{Claim}
\theoremstyle{definition}
\newtheorem{question}[theorem]{Question}
\newtheorem{definition}[theorem]{Definition}

\newcommand{\game}{{\Game}}
\newcommand{\boldpi}[1]{{\boldsymbol{\Pi}^1_{#1}}}
\newcommand{\boldsigma}[1]{{\boldsymbol{\Sigma}^1_{#1}}}

\newcommand{\boldgameclass}[1]{{\game ^{#1}(<\!\omega^2\textnormal{-}\mathbf\Pi{^1_{1}})}}

\newcommand{\bolddelta}[1]{{\boldsymbol{{\delta}}^1_{#1}}}
\newcommand{\boldDelta}[1]{{\boldsymbol{{\Delta}}^1_{#1}}}
\newcommand{\admistwo}[1]{L_{\kappa_3^{#1}}[T_2,#1]}

\newcommand{\WO}{{\textrm{WO}}}

\newcommand{\DEF}{=_{{\textrm{DEF}}}}

\newcommand{\Diff}{\operatorname{Diff}}

\newcommand{\sharpcode}[1]{  \left|   #1  \right|}

\newcommand{\lessthanshort}[1]{{<\! #1}}

\newcommand{\comm}[1]{{}}

\oper{\ucf}{ucf}					
\oper{\sign}{sign}					
\oper{\lh}{lh}

\begin{document}

\title{The higher sharp I: on $M_1^{\#}$}

\author{Yizheng Zhu}

\affil{Institut f\"{u}r mathematische Logik und Grundlagenforschung \\
Fachbereich Mathematik und Informatik\\
Universit\"{a}t M\"{u}nster\\
Einsteinstr. 62 \\
48149 M\"{u}nster, Germany 
}
\maketitle

\begin{abstract}
We establish the descriptive set theoretic representation of the mouse $M_n^{\#}$, which is called $0^{(n+1)\#}$. This part deals with the case $n=1$.
\end{abstract}



\section{Introduction to this series of papers}
\label{sec:introduction}

This is an introduction to a series of four papers. 

The collection of projective subsets of $\mathbb{R}$ is the minimum one which contains all the Borel sets and is closed under both complements and continuous images. Despite its natural-looking definition, many fundamental problems about projective sets are undecidable in ZFC, for instance, if all projective sets are Lebesgue measurable. The axiom of Projective Determinacy (PD)  is the most satisfactory axiom that settles these problems by producing a rich structural theory of the projective sets.
PD implies certain regularity properties of projective sets: all projects of reals are Lebesgue measurable (Mycielski, Swierczkowski), have the Baire property (Banach, Mazur) and are either countable or have a perfect subset (Davis) (cf.~ \cite{mos_dst}). The structural theory of the projective sets are centered at good Suslin representations of projective sets. Moschovakis~ \cite{mos_dst} shows that PD implies the scale property of the pointclasses $\boldpi{2n+1}$ and $\boldsigma{2n+2}$. It follows that there is a nicely behaved tree $T_{2n+1}$ that projects to the good universal $\Sigma^1_{2n+2}$ set. 
So the analysis of $\boldsigma{2n+2}$ sets is reduced to that of the tree $T_{2n+1}$, the canonical model $L[T_{2n+1}]$ and its relativizations. 
The canonicity of $L[T_{2n+1}]$ is justified by Becker-Kechris \cite{becker_kechris_1984} in the sense that $L[T_{2n+1}]$ does not depend on the choice of $T_{2n+1}$. The model $L[T_{2n+1}]$ turns out to have many analogies with $L= L[T_1]$.  These analogies support the generalizations of classical results on $\boldsigma{2}$ sets to $\boldsigma{2n+2}$ sets.

The validity of PD is further justified by Martin-Steel \cite{martin_steel_PD}. They show that PD is a consequence of large cardinals: if there are $n$ Woodin cardinals below a measurable cardinal, then $\boldpi{n+1}$ sets are determined.  Inner model quickly developed into the region of Woodin cardinals. $M_n^{\#}$, the least active mouse with $n$ Woodin cardinals, turns out to have its particular meaning in descriptive set theory.
Martin \cite{martin_largest} (for $n=0$) and Neeman \cite{nee_opt_I,nee_opt_II} (for $n\geq 1$) show that $M_n^{\#}$ is many-one equivalent to the good universal $\game^{n+1} ( \lessthanshort{\omega^2}\text{-}\Pi^1_1)$ real. 
 Steel \cite{steel_hodlr_1995} shows that  $L[T_{2n+1}] = L[M^{\#}_{2n,\infty} | \bolddelta{2n+1}]$ where $M^{\#}_{2n,\infty}$ is the direct limit of all the countable iterates of $M^{\#}_{2n}$, and that $\bolddelta{2n+1}$ is the least cardinal that is strong up to the least Woodin of $M_{2n,\infty}^{\#}$. This precisely explains the analogy between $L[T_{2n+1}]$ and $L$. The mechanism of inner model theory is therefore applicable towards understanding the structure $L[T_{2n+1}]$. 

In this series of  papers, we generalize the Silver indiscernibles for $L$ to the level-($2n+1$) indiscernibles of $L[T_{2n+1}]$. The theory of $L[T_{2n+1}]$ with the level-($2n+1$) indiscernibles will be called $0^{(2n+1){\#}}$, which is many-one equivalent to $M_{2n}^{\#}$. At the level of mice with an odd number of Woodins, $M_{2n-1}^{\#}$ is the optimal real with the basis result for $\Sigma^1_{2n+1}$ sets (cf. \cite[Section 7.2]{steel-handbook}): Every nonempty $\Sigma^1_{2n+1}$ set has a member recursive in $M_{2n-1}^{\#}$.  The basis result for $\Sigma^1_{2n+1}$ was originally investigated in \cite{Q_theory}, with the intention of generalizing Kleene's basis theorem: Every nonempty $\Sigma^1_1$ set of real has a member recursive in Kleene's $\mathcal{O}$. The real $y_{2n+1}$, defined  \cite{Q_theory}, turns out $\Delta^1_{2n+1}$ equivalent to $M_{2n-1}^{\#}$. 
In this series of papers, we define the canonical tree $T_{2n}$ that projects to a good universal $\Pi^1_{2n}$ set. 
It is the natural generalization of the Martin-Solovay tree $T_2$ that projects a good universal $\Pi^1_2$ set. We show that  $L_{\kappa_{2n+1}}[T_{2n}]$, the minimum admissible set over $T_{2n}$, shares most of the standard properties of $L_{\omega_1^{CK}}$, in particular, the higher level analog of the Kechris-Martin theorem \cite{Kechris_Martin_I,Kechris_Martin_II}. We define $0^{(2n){\#}}$ as the set of truth values in $L_{\kappa_{2n+1}}[T_{2n}]$ for formulas of complexity slightly higher than $\Sigma_1$. $0^{(2n){\#}}$ is many-one equivalent to both $M_{2n+1}^{\#}$ and $y_{2n+1}$. Summing up, we have 
\begin{displaymath}
  0^{(n+1)\#} \equiv_m M_n^{\#}.
\end{displaymath}

We start to give a detailed explanation of the influence of the higher sharp in the structural theory of projective sets and in inner model theory. 
The set theoretic structures tied to $\Pi^1_1$ sets are $L_{\omega_1^{CK}}$ and its relativizations. 
The classical results on $\Pi^1_1$ sets and $L_{\omega_1^{CK}}$ include:
\begin{enumerate}
\item (Model theoretic representation of $\Pi^1_1$) $A  \subseteq \mathbb{R}$ is $\Pi^1_1$ iff there is a $\Sigma_1$ formula $\varphi$ such that $x \in A \eqiv L_{\omega_1^x}[x] \models \varphi(x)$. 
\item (Mouse set) $x \in \mathbb{R} \cap L_{\omega_1^{CK} }$ iff $x  $ is $\Delta^1_1$ iff $x$ is $\Delta^1_1$ in a countable ordinal.
\item (The transcendental real over $L_{\omega_1^{CK}}$) $\mathcal{O}$ is the $\Sigma_1$-theory of $L_{\omega_1^{CK}}$. 
\item ($\Pi^1_1$-coding of ordinals below $\omega_1$) $x \in \WO$ iff $x$ codes a wellordering of a subset of $\omega$. Every ordinal below $\omega_1$ is coded by a member of $\WO$. $\WO$ is $\Pi^1_1$. 
\end{enumerate}
$\Sigma^1_2$ sets are $\omega_1$-Suslin via the Shoenfield tree $T_1$. 
The complexity of $T_1$ is essentially that of $\WO$, or $\Pi^1_1$.  The set-theoretic structures in our attention are $L=L[T_1]$ and its relativizations. Assuming every real has a sharp, the classical results related to $L$ include:
\begin{enumerate}
\item (Model theoretic representation of $\Sigma^1_2$) $A  \subseteq \mathbb{R}$ is $\Sigma^1_2$ iff there is a $\Sigma_1$ formula $\varphi$ such that $x \in A \eqiv L[x] \models \varphi(x)$. 
\item (Mouse set) $x \in \mathbb{R} \cap L $ iff $x$ is $\Delta^1_2$ in a countable ordinal.
\item (The transcendental real over $L$) $0^{\#} $ is the theory of $L$ with Silver indiscernibles, or equivalently, the least active sound mouse projecting to $\omega$.
\item ($\Delta^1_3$-coding of ordinals below $u_{\omega}$) $\WO_{\omega}$ is the set of sharp codes. Every ordinal $\alpha <u_{\omega}$ has a sharp code $\corner{\gcode{\tau}, x^{\#}}$ so that $\alpha = \tau^{L[x]}(x,u_1,\dots,u_k)$. The comparison of sharp codes is $\Delta^1_3$. 
\end{enumerate}

Inner model theory start to participate at this level. 
Based on the theory of sharps for reals, the Martin-Solovay tree $T_2$ is defined. $T_2$ is essentially a tree on $u_{\omega}$. The complexity of $T_2$ is $\Delta^1_3$ via the sharp coding of ordinals. 

$\Sigma^1_3$ sets are $u_{\omega}$-Suslin via the Martin-Solovay tree $T_2$. 
The structures tied to $\Pi^1_3$ sets are $L_{\kappa_3}[T_2]$ and its relativizations. The theory at this level is in parallel to $\Pi^1_1$ sets and $L_{\omega_1^{CK}}$:
\begin{enumerate}
\item (Model theoretic representation of $\Pi^1_3$, \cite{Kechris_Martin_I,Kechris_Martin_II}) $A  \subset \mathbb{R}$ is $\Pi^1_3$ iff there is a $\Sigma_1$ formula $\varphi$ such that $x \in A \eqiv L_{\kappa_3^x}[T_2,x] \models \varphi(T_2,x)$. 
\item (Mouse set, \cite{Q_theory,Kechris_Martin_I,Kechris_Martin_II,steel_projective_wo_1995}) $x \in \mathbb{R} \cap L_{\kappa_3}[T_2]$ iff $x $ is $\Delta^1_3$ in a countable ordinal iff $x \in \mathbb{R} \cap M_1^{\#}$. 
\item (The transcendental real over $L_{\kappa_3}[T_2]$, Theorem~\ref{thm:2sharp-equivalent-to-M1sharp})  $M_1^{\#} \equiv_m 0^{2\#}$.
\item ($\Pi^1_3$-coding of ordinals below $\bolddelta{3}$, essentially by Kunen in \cite{sol_delta13coding}) $\WO^{(3)}$ is the set of reals that naturally code a wellordering of $u_{\omega}$. $\WO^{(3)}$ is $\Pi^1_3$. 
\end{enumerate}
In general, if $\Gamma$ is a pointclass, $\alpha$ is an ordinal, and $f : \mathbb{R}  \twoheadrightarrow \alpha$ is a surjection, then $\Code(f) = \set{(x,y)}{f(x) \leq f(y)}$ and $f$ is in $\Gamma$ iff $\Code(f)$ is in $\Gamma$; $\alpha$ is $\Gamma$-wellordered cardinal iff there is a surjection $f: \mathbb{R} \twoheadrightarrow \alpha$ such that $f$ is in $\Gamma$ but there is no $\beta< \alpha$ and surjections $g: \mathbb{R}\twoheadrightarrow   \beta$, $h : \beta \twoheadrightarrow \alpha$ such that both  $g$  and $\set{(x,y)}{f(x) = h \circ g(y)}$ is in $\Gamma$. 
The above list can be continued:
\begin{enumerate}\setcounter{enumi}{4}
\item 
  The uncountable $\boldDelta{3}$ wellordered cardinals are $(u_k:1 \leq k\leq \omega)$. 
\end{enumerate}
The heart of the new knowledge at this level is the equality of pointclass in Theorem~\ref{thm:pointclass_game_exchange}: $\game^2 ( <\! \omega^2 \text{-}\Pi^1_1) = <\!u_{\omega} \text{-}\Pi^1_3.$ This is the main objective of this paper. 
 Philosophically speaking, as $\game^2 \Pi^1_1 = \Pi^1_3$, this equality reduces the ``non-linear'' part $\game^2$ to the ``linear'' part $\lessthanshort{u_{\omega}}$. 
Based on this equality, $0^{2\#}$ is defined to be the set of truth of $L_{\kappa_3}[T_2]$ for formulas of complexity slightly larger than $\Sigma_1$, cf.\ Definitions~\ref{def:OT2x}-\ref{def:2sharp}. $0^{2\#}$ is essentially $y_3$, defined in \cite{Q_theory}. It is a good universal $<\!u_{\omega} \text{-}\Pi^1_3$ subset of $\omega$. The many-one equivalence $M_1^{\#} \equiv_m 0^{2\#}$ is thus obtained using Neeman \cite{nee_opt_I,nee_opt_II}.  Under $AD$, we have $u_k = \aleph_k$, and  \cite{kechris_ad_proj_ord} summarizes the further structural theory at this level. 
The expression of $0^{2\#}$ opens the possibility of running recursion-theoretic arguments in $L_{\kappa_3}[T_2]$ that generalize those in $L_{\omega_1^{CK}}$. 

The Moschovakis tree $T_{2n+1}$ projects to the good universal $\Sigma^1_{2n+2}$ set. 
The structures tied to $\Sigma^1_{2n+2}$ sets are $L[T_{2n+1}]$ and its relativizations. $L[T_{2n+1}]$ is the higher level analog of $L$: 
\begin{enumerate}
\item (Model theoretic representation of $\Sigma^1_{2n+2}$) $A  \subseteq \mathbb{R}$ is $\Sigma^1_{2n+2}$ iff there is a $\Sigma_1$ formula $\varphi$ such that $x \in A \eqiv L[T_{2n+1},x] \models \varphi(T_{2n+1},x)$. 
\item (Mouse set, \cite{steel_projective_wo_1995})  $x \in \mathbb{R} \cap L[T_{2n+1}]$ iff $x $ is $\Delta^1_{2n+2}$ in a countable ordinal iff $x \in \mathbb{R} \cap M_{2n}^{\#}$. 
\item (The transcendental real over $L[T_{2n+1}]$, to be proved in this series of papers)  $M_{2n}^{\#} \equiv_m 0^{(2n+1){\#}}$.
\item ($\Delta^1_{2n+3}$-coding of ordinals below $u^{(2n+1)}_{E(2n+1)}$) $\WO^{(2n+1)}_{E(2n+1)}$ is the set of level-($2n+1$) sharp codes for ordinals in $u^{(2n+1)}_{E(2n+1)}$. The comparison of level-($2n+1$) sharp codes is $\Delta^1_{2n+3}$. 
\end{enumerate}
$0^{(2n+1)\#}$ is the theory of $L[T_{2n+1}]$ with level-($2n+1$) indiscernibles. 
The structure of the level-($2n+1$) indiscernibles is more complicated than their order, as opposed to the order indiscernibles for $L$. The level-($2n+1$) indiscernibles form a tree structure, and the type realized in $L[T_{2n+1}]$ by finitely many of them depends only on the finite tree structure that relates them. This tree structure resembles the structure of measures (under AD) witnessing the homogeneity of $S_{2n+1}$, a tree on $\omega \times \bolddelta{3}$ that projects to the good universal $\Pi^1_{2n+1}$ set. 
We give a purely syntactical definition of $0^{(2n+1)\#}$ as the unique iterable, remarkable, level $\leq 2n$ correct level-($2n+1$) EM blueprint. This is the higher level analog of $0^{\#}$ as the unique wellfounded remarkable EM blueprint. The ``iterability'' part takes the form $\forall^{\mathbb{R}} (\Pi^1_{2n+1} \to \Pi^1_{2n+1})$, making the complexity of the whole definition $\Pi^1_{2n+2}$.
The ordinal $u^{(2n+1)}_{E(2n+1)}$ is a level-$(2n+1)$ uniform indiscernible. It will be discussed in the next paragraph. When $n=0$, $u^{(1)}_{E(1)} = u_{\omega}$.

The structure tied to arbitrary $\Pi^1_{2n+1}$ sets are defined. 
By induction, we have level-($2n-1$) indiscernibles for $L_{\bolddelta{2n-1}}[T_{2n-1}]$ and the real $0^{(2n-1){\#}}$. Based on the EM blueprint formulation of $0^{(2n-1) \#}$, we define the level-$2n$ Martin-Solovay tree $T_{2n}$. It is the higher level analog of $T_2$. 
This is the most canonical tree that enables the correct generalization of the structural theory related to $\Pi^1_{2n+1}$ sets. The structures in our attention are $L_{\kappa_{2n+1}}[T_{2n}]$, the least admissible set over $T_{2n}$, and its relativizations:
\begin{enumerate}
\item (Model theoretic representation of $\Pi^1_{2n+1}$, to be proved in this series of papers) $A  \subseteq \mathbb{R}$ is $\Pi^1_{2n+1}$ iff there is a $\Sigma_1$ formula $\varphi$ such that $x \in A \eqiv L_{\kappa_{2n+1}^x}[T_{2n},x] \models \varphi(T_{2n},x)$. 
\item (Mouse set, \cite{steel_projective_wo_1995}) $x \in \mathbb{R} \cap L_{\kappa_{2n+1}}[T_{2n}]$ iff $x$ is $\Delta^1_{2n+1}$ in a countable ordinal iff $x \in M_{2n-1}^{\#}$. 
\item (The transcendental real over $L_{\kappa_{2n+1}}[T_{2n}] $, to be proved in this series of papers) $M_{2n-1}^{\#} \equiv_m 0^{(2n)\#}$.
\item ($\Pi^1_{2n+1}$-coding of ordinals below $\bolddelta{2n+1}$) $\WO^{(2n+1)}$ is the set of reals that naturally code a wellordering of $u^{(2n-1)}_{E(2n-1)}$. $\WO^{(2n+1)}$ is $\Pi^1_{2n+1}$. 
\item 
  The uncountable $\boldDelta{2n+1}$ wellordered cardinals are $(u_k:1 \leq k\leq \omega)$, $(u^{(3)}_{\xi} : 1 \leq \xi \leq E(3))$, $\dots$, $(u^{(2n-1)}_{\xi} : 1 \leq \xi \leq E(2n-1))$, where $E(0) = 1$, $E(i+1) = \omega^{E(i)}$ via ordinal exponentiation. 
\end{enumerate}
The equivalence $M_{2n-1}^{\#} \equiv_m 0^{(2n)\#}$ will be based on the equality of pointclasses: $\game^{2n}(\lessthanshort{\omega^2}\text{-}\Pi^1_1) = \lessthanshort{u^{(2n-1)}_{E(2n-1)}}\text{-}\Pi^1_{2n+1}$.
$\set{u^{(2n-1)}_{\xi}}{1 \leq \xi \leq E(2n-1)}$ is the set of level-$(2n-1)$ uniform indiscernibles. It is the higher level analog of the first $\omega+1$ uniform indiscernibles $\set{u_n}{ 1 \leq n \leq \omega}$.  
Under full AD, the uncountable $\boldDelta{2n+1}$ wellordered cardinals enumerate all the uncountable cardinals below $\bolddelta{2n+1}$: $u_k = \aleph_k$ for $1 \leq k<\omega$, $u^{(2i+1)}_{\xi} = \aleph_{E(2i-1)+\xi}$ for $1 \leq \xi \leq E(2i+1)$. Assume AD for the moment. 
The equation $\bolddelta{2n+1} = \aleph_{E(2n-1)+1}$ is originally proved by Jackson in \cite{jackson_delta15,jackson_proj_ord}. Jackson shows that every successor cardinal in the interval $[\bolddelta{2n-1}, \aleph_{E(2n-1)})$ is the image of $\bolddelta{2n-1}$ via an ultrapower map induced by a measure on $\bolddelta{2n-1}$. 
 \cite{jk_desc} goes on to show that for a certain collection of measures $\mu$ on $\bolddelta{3}$, 
every description leads to a canonical function representing a cardinal modulo $\mu$.
 \cite{jk_desc,jackson_loewe_2013} compute the cofinality of the cardinals below $\bolddelta{\omega}$.  In this series of papers, we demonstrate the greater importance of the set theoretic \emph{structures} tied to these cardinals over their \emph{order type}. It is the inner model $L[T_{2n-1}]$ and its images via different ultrapower maps that give birth to the uncanny order type $E(2n-1)+1$. 
The level-$(2n-1)$ uniform indiscernibles $(u^{(2n-1)}_{\xi} : 1 \leq \xi \leq E(2n-1))$ are defined under this circumstance. Recall that the first $\omega$ uniform indiscernibles can be generated by $j^{\mu^n}(L_{\omega_1}) = L_{u_{n+1}}$, where $\mu^n$ is the $n$-fold product of the club measure on $\omega_1$; if $1 \leq i \leq {n+1}$, then $u_i$ is represented modulo $\mu^n$ by a projection map; every ordinal below $u_{n+1}$ is in the Skolem hull of $\se{x , u_1,\dots,u_n}$ over $L[x]$ for some $x \in \mathbb{R}$. This scenario is generalized by the level-$(2n-1)$ uniform indiscernibles. As a by-product, we simplify the arguments in \cite{jackson_delta15,jackson_proj_ord,jk_desc,jackson_loewe_2013},  show in full generality that any description represents a cardinal modulo any measure on $\bolddelta{2n-1}$, and establish the effective version of the cofinality computations. 

The whole argument is inductive. Assume AD for simplicity. In the computation of $\bolddelta{2n+1}$ in \cite{jackson_delta15,jackson_proj_ord}, the strong partition property of $\bolddelta{2n+1}$ is proved and used inductively in the process. 
Our argument reproves the strong partition property of $\bolddelta{2n+1}$ using the EM blueprint formulation of $0^{(2n+1)\#}$. 
The definition of $0^{(2n+1)\#}$ is based on the analysis of level-($2n+1$) indiscernibles, whose existence depend on the homogeneous Suslin representations of $\boldpi{2n}$ sets, which in turn follow from the strong partition property of $\bolddelta{2n-1}$. Just as the main ideas of the computation of $\bolddelta{2n+1}$ boil down to that of $\bolddelta{5}$, all the non-trivialities are contained in the first few levels. The general inductive step is merely a technical manifestation. 

A deeper insight into the interaction between inner model theory and Jackson's computation of projective ordinals in \cite{jackson_delta15,jackson_proj_ord} is the concrete information on the direct system of countable iterates of $M_{2n}^{\#}$. Put $n=1$ and assume AD for simplicity sake. Put $M_{2,\infty}^{-} = L_{\bolddelta{3}}[T_3]$. We define
$(c^{(3)}_{\xi}: \xi < \bolddelta{3})$, a continuous sequence in $\bolddelta{3}$ that generates the set of level-3 indiscernibles for $M_{2,\infty}^{-}$. 
Each $M_{2,\infty}^{-}|c^{(3)}_{\xi}$ is the direct limit of $\Pi^1_3$-iterable mice whose Dodd-Jensen order is $c^{(3)}_{\xi}$. 
We define an alternative direct limit system indexed by ordinals in $u_{\omega}$ which is dense in the system leading to $M_{2,\infty}^{-}|c^{(3)}_{\xi}$. 
The advantage of this dense subsystem is that it leads to a good coding of $M_{2,\infty}^{-}| c^{(3)}_{\xi}$ by a subset of $u_{\omega}$. 
The indexing ordinals are represented by wellorderings on $\omega_1$ of order type $\omega_1+1$ modulo measures on $\omega_1$  arising from the strong partition property on $\omega_1$. 
Any order-preserving injection between two such wellorderings corresponds to an elementary embedding between models of this new direct limit.  This injection is an isomorphism just in case its corresponding elementary embedding is essentially an iteration map, i.e., commutes with the comparison maps. The new direct system is then guided by isomorphisms between wellorderings on $\omega_1$ of order type $\omega_1+1$. In this regard, the Dodd-Jensen property of mice corresponds to the simple fact that if $f$ is an order preserving map between ordinals, then $\alpha \leq f(\alpha)$ pointwise. This observation is not surprising at all, as the Dodd-Jensen property on iterates of $0^{\#}$ \emph{is} originated from this simple fact.  
This viewpoint might be a prelude to understanding the combinatorial nature of iteration trees on mice with finitely many Woodin cardinals.

A key step in computing the upper bound of $\bolddelta{5}$ in \cite{jackson_delta15} is the (level-3) Martin tree. For the reader familiar with the Martin tree and the purely descriptive set theoretical proof of the Kechris-Martin theorem in \cite[Section 4.4]{jackson_handbook}, the level-1 version of the Martin tree is essentially an analysis of partially iterable sharps. The level-3 Martin tree is therefore replaced by an analysis of partially iterable level-3 sharps in this series of papers. The aforementioned new direct limit system indexed by ordinals in $u_{\omega}$ applies to any partially iterable mouse, so that its possibly illfounded direct limit is naturally coded by a subset of $u_{\omega}$.  This is yet another incidence that descriptive set theory and inner model theory are two sides of the same coin.

Apart from inner model theory, the pure computational component in \cite{jackson_delta15,jackson_proj_ord} has a major simplification. Under AD, a successor cardinal in the interval $[\bolddelta{3}, \aleph_{\omega^{\omega^{\omega}}})$ is represented by a measure $\mu$ on $\bolddelta{3}$ and a description. The original definition of description involves a finite iteration of ultrapowers on $u_{\omega}$. The ``finite iteration of ultrapowers'' part is now simplified to a single ultrapower. The analysis will be in the third paper of this series. 

As $L_{\kappa_{2n+1}}[T_{2n}]$ is the correct structure tied to $\Pi^1_{2n+1}$ sets, it is natural to investigate its intrinsic structure. However, little is known at this very step. The closest result is on the full model $L[T_{2n}]$. The uniqueness of $L[T_{2n}]$ is proved by Hjorth \cite{hjorth_LT2} for $n=1$ and Atmai \cite{atmai_thesis} for general $n$. Here, uniqueness means that if $T'$ is the tree of another $\Delta^1_{2n+1}$-scale on a good universal $\Pi^1_{2n}$ set, then $L[T_{2n}] = L[T']$. Atmai-Sargsyan \cite{atmai_thesis} goes on to show that the full model $L[T_{2n}]$ is just $L[M_{2n-1,\infty}^{\#}]$, where $M_{2n-1,\infty}^{\#}$ is the direct limit of all the countable iterates of $M_{2n-1}^{\#}$. A test question that separates $L_{\kappa_3}[T_2]$ from $L[T_2]$ is the inner model theoretic characterization of $C_3$, the largest countable $\Pi^1_3$ set: if $x \in C_3$, must $x$ be $\Delta^1_3$-equivalent to a master code in $M_2$? (cf. \cite[p.13]{cabal1_intro_steel}) 
We will set up a good preparation for tackling this problem. 

Looking higher up, the technique in this series of papers should generalize to arbitrary projective-like pointclasses in $L(\mathbb{R})$ and beyond. The descriptive set theory counterpart of larger mice should enhance our understanding of large cardinals. Typical open questions in the higher level include: 
\begin{enumerate}
\item (cf. \cite[Problem 19]{aim_2004}) Assume AD. Let $\boldsymbol{\Gamma}$ be a $\boldpi{1}$-like scaled pointclass (i.e., closed under $\forall^{\mathbb{R}}$, continuous preimages and non-self-dual) and Let $\boldsymbol{\Delta} = \boldsymbol{\Gamma} \cap {\boldsymbol{\Gamma}}^{\smallsmile}$, $\boldsymbol{\delta} = \sup \set{\sharpcode{<}}{< \text{ is a prewellordering in } \boldsymbol{\Delta} }$. Is $\boldsymbol{\Gamma}$ closed under unions of length $<\boldsymbol{\delta}$?
\item Assume AD. Let $\boldsymbol{\Gamma}$, $\boldsymbol{\delta}$ be as in 1. Must $\boldsymbol{\delta}$ have the strong partition property?
\item Assume AD. If $\kappa \leq \lambda$ are cardinals, must $ \cf(\kappa^{++}) \leq \cf(\lambda^{++})$?
\end{enumerate}

We now switch to some immediate applications on the theory of higher level indiscernibles. Our belief is that any result in set theory that involves sharp and Silver indiscernibles should generalize to arbitrary projective levels. 

Woodin \cite{SUW} proves that boldface $\boldpi{2n+1}$-determinacy is equivalent to ``for any real $x$, there is an $(\omega,\omega_1)$-iterable $M_{2n}^{\#}(x)$''. The lightface scenario is tricky however.
Woodin (in unpublished work for odd $n$) and Neeman \cite{nee_opt_I,nee_opt_II} proves that the existence of an $\omega_1$-iterable $M_n^{\#}$ implies boldface $\boldpi{n}$-determinacy and lightface $\Pi^1_{n+1}$-determinacy. 

\begin{question}[{cf. \cite[\#9]{cabal_4_survey}}]
\label{question:det}
  Assume $\boldpi{n}$-determinacy and $\Pi^1_{n+1}$-determinacy. Must there exist an $\omega_1$-iterable $M_{n}^{\#}$? 
\end{question}
Note that the assumption of boldface $\boldpi{n}$-determinacy in Question~\ref{question:det} is necessary, as $\Delta^1_2$-determinacy alone is enough to imply that there is a model of OD-determinacy (Kechris-Solovay \cite{kec_sol_1985}). The cases $n \in \se{0,1}$ in Question~\ref{question:det} are solved positively by Harrington in \cite{harr_sharp_det} and by Woodin in \cite{hod_as_a_core_model}.  The proof of the $n=1$ case heavily relies on the theory of Silver indiscernibles for $L$. The theory of level-3 indiscernibles for $L_{\bolddelta{3}}[T_3]$ is thus involved in proving the general case when $n$ is odd. 

\begin{theorem}
  \label{thm:determinacy_to_mice}
  Assume $\boldpi{2n+1}$-determinacy and $\Pi^1_{2n+2}$-determinacy. Then there exists an $(\omega,\omega_1)$-iterable $M_{2n+1}^{\#}$. 
\end{theorem}
The proof of Theorem~\ref{thm:determinacy_to_mice} will appear in further publications. 
The case $n\geq 2$ even in Question~\ref{question:det} remains open.

Another application is the $\delta$-ordinal of intermediate pointclasses between $\boldpi{m}$ and $\boldDelta{m+1}$. 
If $\Gamma$ is a pointclass, $\delta(\Gamma)$ is the supremum of the lengths of $\Gamma$-prewellorderings on $\mathbb{R}$.
 $A \subseteq \mathbb{R}$ is $\Gamma_{m,n}(z)$ iff for some formula $\psi$ we have $x \in A \eqiv M_{m-1}[x,z] \models \psi (x,z, \aleph_1,\dots,\aleph_n)$. $A$ is  $\boldsymbol{\Gamma}_{m,n}$ iff $A $ is ${\Gamma}_{m,n}(z)$ for some real $z$. Hjorth \cite{hjorth_boundedness_lemma} proves that $\delta(\boldsymbol{\Gamma}_{1,n}) = u_{n+2}$ under $\boldDelta{2}$-determinacy. Sargsyan \cite{sarg_pwo_2013}
proves that under AD, $\sup_{n<\omega} \delta (\boldsymbol{\Gamma}_{2k+1,n}) $ is the cardinal predecessor of $\bolddelta{2k+3}$. The exact value of $\delta(\boldsymbol{\Gamma}_{2k+1,n})$ remains unknown. Based on the theory of higher level indiscernibles, we can define the pointclasses $\boldsymbol{\Lambda}_{2k+1,\xi}$ for $0<\xi \leq  E(2k+1)$. For the moment we need the notations in this series of papers. 
 $A \subseteq \mathbb{R}$ is $\Lambda_{3, \xi+1}(z)$ iff for some level-3 tree $R$ such that $\llbracket \emptyset \rrbracket_R = \widehat{\xi}$, for some $\mathcal{L}^{\underline{x},R}$-formula $\psi$ we have $x \in A \eqiv \gcode{\psi} \in (x,z)^{3\#}(R)$. When $\xi$ is a limit, $\Lambda_{3,\xi}(z) = \bigcup_{\eta < \xi}\Lambda_{3,\eta}(z)$. 
$A$ is $\boldsymbol{\Lambda}_{3,\xi}$ iff $A$ is $\Lambda_{3,\xi}(z)$ for some real $z$. 
 \begin{theorem}
   \label{thm:generalize_hjorth}
   Assume $\boldDelta{4}$-determinacy and  $0<\xi < \omega^{\omega^{\omega}}$. If $\xi$ is a successor ordinal, then $\delta(\boldsymbol{\Lambda}_{3,\xi}) = u^{(3)}_{\xi+1}$. If $\xi $ is a limit ordinal, then $\delta(\boldsymbol{\Lambda}_{3,\xi}) = u^{(3)}_{\xi}$. 
 \end{theorem}
The proof of Theorem~\ref{thm:generalize_hjorth} and its higher level analog will appear in further publications. The question on the value of $\delta(\boldsymbol{\Gamma}_{3,n})$ is then reduced to the relative position of $\boldsymbol{\Gamma}_{3,n}$ in the hierarchy $(\boldsymbol{\Lambda}_{3,\xi} : 0 < \xi < \omega^{\omega^{\omega}})$. The results of this series of papers combined with Neeman \cite{nee_opt_I,nee_opt_II} yield the following estimate:
\begin{displaymath}
  \boldsymbol{\Lambda}_{3,\omega^{\omega^{n}}} \subseteq   \boldsymbol{\Gamma}_{3,n} \subseteq \boldsymbol{\Lambda}_{3,\omega^{\omega^{n+1}}+1}.
\end{displaymath}
We conjecture that $
  \boldsymbol{\Lambda}_{3,\omega^{\omega^{n+1}}} \subsetneq   \boldsymbol{\Gamma}_{3,n} \subsetneq \boldsymbol{\Lambda}_{3,\omega^{\omega^{n+1}}+1}$ and $\delta(\boldsymbol{\Gamma}_{3,n} ) = u^{(3)}_{\omega^{\omega^{n+1}}+1}$.

We try to make this series of papers as self-contained as possible. 
The reader is assumed to have some minimum background knowledge in descriptive set theory and inner model theory.
On the descriptive set theory side, we assume basic knowledge of determinacy, scale and its tree representation, homogeneous tree and its ultrapower representation, and at least the results of Moschovakis periodicity theorems. We will briefly recall them in Section~\ref{sec:KM}. Theorem~\ref{thm:BK-KM} by Becker-Kechris \cite{becker_kechris_1984} and Kechris-Martin \cite{Kechris_Martin_I,Kechris_Martin_II} will basically be treated as a black box. 
Knowing its proof would help, though not necessary. 
On the inner model theory side, we assume basic knowledge of mice and iteration trees in the region of finitely many Woodin cardinals, especially Theorem 6.10 in  \cite{steel-handbook}. The level-wise projective complexity associated to mice will be recalled in the second paper of this series. Steel's computation of $L[T_{2n+1}]$ in \cite{steel_hodlr_1995} will be treated as a black box. 
In particular, we require absolutely no knowledge of Jackson's analysis in \cite{jackson_delta15,jackson_proj_ord}. 

This series of papers is organized as follows. This paper is the first one of this series, establishing the many-one equivalence of $0^{2\#}$ and $M_1^{\#}$. The second paper will define $0^{3\#}$ as the theory of $L_{\bolddelta{3}}[T_3]$ with its ``level-3 indiscernibles'' and prove the many-one equivalence of $0^{3\#}$ and $M_2^{\#}$. The third paper will gives a $\Pi^1_4$-axiomatization of the real $0^{3\#}$, define the level-4 Martin-Solovay tree and prove the level-4 Kechris-Martin theorem, which prepares for the induction into the next level. The fourth paper will deal with the general inductive step in the projective hierarchy. 

We arrange these papers in such a way in order to minimize the background knowledge of the first and second papers. The first paper uses only sharps for reals and blackboxed Kechris-Martin. The second paper will introduce homogeneous trees with restricted complexity without too much technicalities that are enough to define $0^{3\#}$. The fine analysis associated to homogeneous trees, especially to the generalized Jackson's analysis, will only show up in the third paper. The fourth paper will be pretty much a routine generalization of the first three.


\section{Backgrounds and preliminaries}
\label{sec:KM}

\subsection{Basic descriptive set theory}
\label{sec:basic-descr-set}

Following the usual treatment in descriptive set theory, $\mathbb{R} = \omega^{\omega}$ is the Baire space, which is homeomorphic to the irrationals of the real line. If $A \subseteq \mathbb{R} \times X$, then $y \in \exists^{\mathbb{R}} A $ iff $\exists x \in \mathbb{R}~ (x,y) \in A$, $y \in \forall^{\mathbb{R}} A$ iff $\forall x \in \mathbb{R}~ (x,y) \in A$, $y \in \game A$ iff Player I has a winning strategy in the game with output $A_y \DEF \set{x}{(x,y) \in A}$. $\game^{n+1}A = \game (\game^n(A))$ when $A$ is a subset of an appropriate product space. 
A pointclass is a collection of subsets of Polish spaces (typically finite products of $\omega$ and $\mathbb{R}$).  
If $\Gamma$ is a pointclass, then $\exists^{\mathbb{R}} \Gamma = \set{\exists^{\mathbb{R}}A}{A \in \Gamma}$, and similarly for $\forall^{\mathbb{R}} \Gamma$, $\game \Gamma$, $\game^n \Gamma$.
$\boldsymbol{\Sigma}^0_1 = \boldsymbol{\Sigma}^1_0$ is the pointclass of open sets. $\Sigma^0_1 = \Sigma^1_0$ is the pointclass of effectively open sets. $\boldpi{n+1} = \forall^{\mathbb{R}} \boldsigma{n}$, $\boldsigma{n+1} = \exists^{\mathbb{R}} \boldpi{n}$, $\Pi^1_{n+1} = \forall^{\mathbb{R}} \Sigma^1_n$, $\Sigma^1_{n+1} = \exists^{\mathbb{R}} \Pi^1_n$.

If $\alpha$ is an ordinal and $A \subseteq \alpha \times X$, then
\begin{displaymath}
x \in \Diff A \eqiv  \exists i < \alpha~(i \text{ is odd }\wedge \forall j < i ( (j,x) \in A) \wedge (i,x) \notin A).
\end{displaymath}
If $\alpha < \omega_1^{CK}$ then $A \subseteq X$ is $ \alpha\text{-}\Pi^1_1$ iff $A = \Diff B$ for some $\Pi^1_1$ $B \subseteq \alpha \times X$. $A$ is $\lessthanshort{\alpha}\text{-}\Pi^1_1$ iff $A$ is $\beta\text{-}\Pi^1_1$ for some $\beta < \alpha$. Martin \cite{martin_largest} proves that $\Pi^1_1$-determinacy implies  $\lessthanshort{\omega^2}\text{-}\Pi^1_1$-determinacy. 

A tree on $ X$ is a subset of $X^{<\omega}$ closed under initial segments. If $T$ is a tree on $X$, $[T] $ is the set of infinite branches of $T$, i.e., $x \in T$ iff $\forall n ~ (x \res n) \in T$. If $T$ is a tree on $\lambda$, $\lambda$ is an ordinal, $[T] \neq \emptyset$, the leftmost branch is $x \in [T]$ such that for any $y \in [T]$, $(x(0),x(1),\dots) $ is lexicographically smaller than or equal to $(y(0),y(1),\dots)$. In addition, if  $x \in [T]$ and for any $y \in [T]$ we have $\forall n ~x(n) \leq y(n)$, then $x$ is the honest leftmost branch of $T$.
A tree $T$ on $\omega \times X$ is identified with a subset of  $\omega^{<\omega} \times X^{<\omega}$ consisting of $(s,t)$ so that $\lh(s) = \lh(t)$ and $((s(i),t(i)))_{i < \lh(s)} \in T$. If $T$ is a tree on $\omega \times X$, $[T]  \subseteq \omega^{<\omega} \times X^{<\omega}$ is the set of infinite branches of $T$.  $p[T]= \set{x}{\exists y ~ (x,y) \in [T]}$ is the projection of $T$. If $T$ is a tree on $\omega \times \lambda$ and $p[T] \neq \emptyset$, then $x$ is the leftmost real of $T$ iff $\exists \vec{\alpha} ~(x,\vec{\alpha})$ is the leftmost branch of $T$.

 \subsection{The Martin-Solovay tree}
\label{sec:effect-kunen-martin}

We assume $\boldpi{1}$-determinacy. This is equivalent to  $\forall x \in \mathbb{R} (x^{\#} \text{ exists})$ by Martin \cite{martin_pi11_det} and  Harrington \cite{harr_sharp_det}. 

$\gamma$ is a uniform indiscernible iff for every $x\in \mathbb{R}$, $\gamma$ is an $x$-indiscernible. The uniform indiscernibles form a club in $\ord$, which are listed $u_1,u_2,\ldots$ in the increasing order. In particular, $u_1 = \omega_1$ and $u_{\omega} = \sup_{n<\omega}u_n$.


The set $ \set{x^{\#}}{x \in \mathbb{R}}$ is $\Pi^1_2$. 
$\WO=\WO_1$ is the set of codes for countable ordinals.
For $1 \leq m < \omega$, $\WO_{m+1}$ is the set of $\corner{\gcode{\tau}, x^{\#}}$ where $\tau$ is an $(m+1)$-ary Skolem term for an ordinal in the language of set theory 
and $x \in \mathbb{R}$. 
The ordinal coded by  $w=\corner{\gcode{\tau}, x^{\#}} \in \WO_{m+1}$ is
\begin{displaymath}
|w| = \tau^{L[x]}(x,u_1,\ldots,u_m).
\end{displaymath}
Every ordinal in $u_{m+1}$ is of the form $\sharpcode{w}$ for some $w \in \WO_{m+1}$. For each $1 \leq m < \omega$, 
\begin{displaymath}
  \set{\tau^{L[x]}(x,u_m)}{ \corner{\gcode{\tau}, x^{\#}} \in \WO_2}
\end{displaymath}
is a cofinal subset of $u_{m+1}$. 
$\WO_{\omega} = \bigcup_{1 < m < \omega} \WO_m$.   
 $\WO$ is $\Pi^1_1$, and $\WO_{m+1}$ is $\Pi^1_2$ for $1 \leq m < \omega$. 
 If $\sigma : \se{1,\dots,m} \to \se{1,\dots,n}$ is order preserving, then define
 \begin{displaymath}
   j^{\sigma} ( \tau^{L[x]}(x, u_1,\dots,u_m) ) = \tau^{L[x]} (x, u_{\sigma(1)},\dots,u_{\sigma(m)}). 
 \end{displaymath}

If $\mathcal{X}$ is a Polish space, $A \subseteq \mathcal{X} \times u_{\omega}$ and $\Gamma$ is a pointclass, say that $A$ is in $\Gamma$ iff
\begin{displaymath}
  A^{*} = \set{ (x, w) }{ x \in \WO_{\omega} \wedge (x, \sharpcode{w} ) \in A}
\end{displaymath}
is in $\Gamma$. $\Gamma$ acting on product spaces are similarly defined. 

$T_2$, defined in   \cite{becker_kechris_1984,Kechris_Martin_II,martin_solovay_basis_1969}, refers to the Martin-Solovay tree on $\omega \times u_{\omega}$ that projects to $ \set{x^{\#}}{x \in \mathbb{R}}$, giving the scale
\begin{displaymath}
  \varphi_{\gcode{\tau}}(x ^{\#}) = \tau^{L[x]} (x,u_1,\ldots,u_{k_{\tau}}),
\end{displaymath}
where $\gcode{\tau}$ is the G\"{o}del number of $\tau$, $\tau$ is $k_{\tau}+1$-ary.
$T_2$ is a $\Delta^1_3$ subset of $(\omega \times u_{\omega})^{<\omega}$. From $T_2$ one can compute a tree $\widehat{T}_2$ on $\omega \times u_{\omega}$ that projects to a good universal $\boldpi{2}$ set. 

To conclude this section, we define the Martin-Solovay tree $T_2$ projecting to $\set{x^{\#}}{x \in \mathbb{R}}$ and its variant $\widehat{T}_2$ projecting to a good universal $\Pi^1_2$ set. This formulation of $T_2$ and $\widehat{T}_2$ will generalize to the higher levels in this series of papers. Let $T\subseteq 2^{<\omega}$ be a recursive tree such that $[T]$ is the set of remarkable EM blueprints over some real. Here we have fixed in advance an effective G\"{o}del coding of first order formulas in the language  $\se{\underline{\in},\underline{x},\underline{c_n}:n<\omega}$, so that an infinite string $x \in 2^{\omega}$ represents the theory $\set{\varphi}{x_{\gcode{\varphi}}=0}$. Fix an effective list of Skolem terms $(\tau_k)_{k<\omega}$ in the language of set theory, where $\tau_k$ is $f(k)+1$-ary, $f$ is effective. $T_2$ is defined as a tree on $2 \times u_{\omega}$ where
\begin{displaymath}
  ( s, (\alpha_0,\ldots,\alpha_{n-1}) ) \in T_2
\end{displaymath}
iff $s \in T$, $\lh(s) = n$, and for any $k,l < n$, for any order preserving $\sigma : \se{1,\ldots,f(k)} \to \se{1,\ldots,f(l)}$, 
  \begin{enumerate}
  \item if ``$ \tau_k (\underline{x},\underline{c_{\sigma(1)}}, \ldots,\underline{c_{\sigma(f(k))}}) = \tau_l(\underline{x},\underline{c_1},\ldots,\underline{c_{f(l)}})$'' is true in $s$, then $j^{\sigma}(\alpha_k) = \alpha_l$;  
\item if ``$ \tau_k (\underline{x},\underline{c_{\sigma(1)}}, \ldots,\underline{c_{\sigma(f(k))}}) < \tau_l(\underline{x},\underline{c_1},\ldots,\underline{c_{f(l)}})$'' is true in $s$, then $j^{\sigma}(\alpha_k) < \alpha_l$;
  \end{enumerate}
In essence, the second coordinate of $T_2$ attempts to  verify the wellfoundedness of the EM blueprint coded in the first coordinate. From $T_2$ we compute $\widehat{T}_2$,   a tree on $\omega \times(\omega \times  u_{\omega})$ that projects to a good universal $\Pi^1_2$ set. By Shoenfield absoluteness, if $\varphi(v)$ is a $\Pi^1_2$ formula, effectively from $\gcode{\varphi}$ we can compute a unary Skolem term  $\tau_{\gcode{\varphi}}$ such that $\tau_{\gcode{\varphi}}^{L[x]} (x)= 0$ iff $\varphi (x)$ holds. Define $(\gcode{\varphi} \concat (v), (s, \vec{\alpha})) \in \widehat{T}_2$ iff $(s, \vec{\alpha}) \in T_2$ and
\begin{enumerate}
\item if ``$\underline{x}(m)=n$'' is true in $s$, then $v(m)=n$;
\item  ``$\tau_{\gcode{\varphi}} (\underline{x})\neq 0$'' is not true in $s$.
\end{enumerate}
So $p[\widehat{T}_2] = \set{\gcode{\varphi} \concat (x)}{ \varphi(x)}$.

\subsection{Q-theory}
\label{sec:q-theory}

From now on until the end of this paper, we assume $\boldDelta{2}$-determinacy. By Kechris-Woodin \cite{KW_det_transfer},  $\boldgameclass{}$-determinacy follows. By Neeman \cite{nee_opt_I,nee_opt_II} and Woodin \cite{hod_as_a_core_model,woodin-handbook}, this is also equivalent to 
``for every $x\in \mathbb{R}$, there is an $(\omega, \omega_1)$-iterable $M_1^{\#}(x)$''.

For $x\in \mathbb{R}$, $\admistwo{x}$ is the minimum admissible set containing $(T_2,x)$. $\kappa_3^x$ is the higher level analog of $\omega_1^x$, the least $x$-admissible. 
The fact that the $\widehat{T}_2$ projects to a good universal $\Pi^1_2$ set implies for every $\Pi^1_3$ set of reals $A$, there is a $\Sigma_1$-formula $\varphi$ such that $x \in A$ iff ${\admistwo{x}\models\varphi(T_2,x) }$; $\varphi$ can be effectively computed from the definition of $A$.  Becker-Kechris in  \cite{becker_kechris_1984} strengthens this fact by allowing a parameter in $u_{\omega}$. The converse direction is shown by Kechris-Martin in \cite{Kechris_Martin_I,Kechris_Martin_II}. The back-and-forth conversion is concluded in \cite{becker_kechris_1984}. 
\begin{theorem}[Becker-Kechris-Martin]\label{thm:BK-KM}
  Assume $\boldDelta{2}$-determinacy. Then for each $A \subseteq u_{\omega} \times \mathbb{R}$, the following are equivalent.
  \begin{enumerate}
  \item $A$ is $\Pi^1_3$.
  \item There is a $\Sigma_1$ formula $\varphi$ such that $(\alpha,x) \in A$ iff $\admistwo{x}\models \varphi(T_2,\alpha,x)$.
  \end{enumerate}
\end{theorem}

The conversions between the $\Pi^1_3$ definition of $A$ and the $\Sigma_1$-formula $\varphi$ are effective.

\section{The equivalence of $x^{2\#}$ and $M_1^{\#}(x)$}
\label{sec:definition-x-2-sharp}

 \begin{definition}
   \label{def:diff-Pi-1-3}
Suppose $\mathcal{X} =  \omega^k \times \mathbb{R}^l$ is a product space. Suppose $x$ is a real and $\beta\leq u_{\omega}$. A subset $A\subseteq \mathcal{X}$ is  $\beta\textnormal{-}\Pi^1_3(x)$  iff there is a $\Pi^1_3(x)$ set $B \subseteq u_{\omega} \times \mathcal{X}  $ such that $A = \Diff B$.
$A$ is $\beta\textnormal{-}\Pi^1_3$ iff $A$ is $\beta\textnormal{-}\Pi^1_3(0)$. $A$ is $\beta\textnormal{-}\boldpi{3}$ iff $A$ is $\beta\textnormal{-}\Pi^1_3(x)$ for some real $x$.
 \end{definition}

By Theorem~\ref{thm:BK-KM}, when $\beta$ is a limit ordinal,  $A\subseteq \mathcal{X}$ is  $\beta\textnormal{-}\Pi^1_3(x)$  iff there is a pair of  $\Sigma_1$-formulas $(\varphi,\psi)$ such that
\begin{displaymath}
 (\vec{n},\vec{y}) =  (n_1,\ldots,n_k,y_1,\ldots,y_l)\in A 
\end{displaymath}
iff
\begin{displaymath}
\admistwo{x,\vec{y}} \models ~\exists \alpha<\beta (\forall \eta < \alpha~ \varphi(\eta,\vec{n},\vec{y},T_2,x) \wedge \neg \psi(\alpha,\vec{n},\vec{y},T_2,x)).
\end{displaymath}


\begin{lemma}
  \label{lem:diff-Pi-1-3-to-game}
Assume $\boldDelta{2}$-determinacy. Suppose $ n,m$ are positive integers. If $A$ is $(u_n)^m\textnormal{-}\Pi^1_3(x)$, then $A$ is $\game^2(\omega n \textnormal{-} \Pi^1_1(x))$. 
\end{lemma}
\begin{proof}
  Without loss of generality, we assume $x=0$, $m=1$, and $A\subseteq \mathbb{R}$. Let $B$ be a $\Pi^1_3$ subset of $ u_{\omega} \times \mathbb{R} $  such that $A = \Diff B$. 
Let $B^{*} = \set{(w,y)}{(\left| w \right|,y) \in B}$ be the $\Pi^1_3$ code set  of $B$. 
Let $C\subseteq   \mathbb{R}^3$ be a $\Sigma^1_2$ set such that
\begin{displaymath}
  (w,y) \in B^{*} \eqiv \forall  r (w,y,r) \in C.
\end{displaymath}

Consider the game $H(y)$, where I produces $w,r\in \mathbb{R}$,  II  produces  $w',r'\in \mathbb{R}$. The game is won by I iff both of the following hold:
\begin{enumerate}
\item $w\in \WO_n$,  $\left| w \right|$ is odd, and $(w,y,r) \notin C $.
\item If $w'\in \WO_n$,  $\left| w' \right|$ is even, and $(w',y,r') \notin C$, then $\left| w \right|< \left| w' \right|$.
\end{enumerate}
Therefore, $y \in A$ iff I has a winning strategy in $H(y)$. 

Since $L[y,w,r,w',r']$ is $\Sigma^1_2$-absolute, and since the relation $\left| w \right| \leq \left| w' \right|$ for $w, w' \in \WO_n$ is  definable over $L[y,w,r,w',r']$ from parameters $u_1,\ldots,u_{n-1}$, the payoff set of the game $H(y)$ can be expressed as a first order statement over $L[y,\cdot]$ from parameters $u_1,\ldots,u_{n-1}$. That is, there is a formula $\theta$ such that an infinite run
\begin{displaymath}
(w,r,w',r')
\end{displaymath}
is won by I iff
\begin{displaymath}
L[y,w,r,w',r'] \models \theta(y,w,r,w',r',u_1,\ldots,u_{n-1}).
\end{displaymath}
It follows by Martin \cite{martin_largest} that the payoff set of $H(y)$ is $\game(\omega n\textnormal{-}\Pi^1_1(y))$, uniformly in $y$, hence determined. Hence $A$ is in $\game^2(\omega n\textnormal{-}\Pi^1_1)$.
\end{proof}

\begin{lemma}
  \label{lem:game-to-diff-Pi-1-3}
Assume $\boldDelta{2}$-determinacy. 
Let $ n<\omega$. If $A$ is $\game^2(\omega n \textnormal{-}\Pi^1_1(x))$, then $A$ is $u_{n+2}\textnormal{-}\Pi^1_3(x)$.
\end{lemma}
\begin{proof}
Without loss of generality, assume $x=0$ and $A \subseteq \mathbb{R}$.   We produce an effective transformation from a $\game^2(\omega n \textnormal{-}\Pi^1_1)$ definition to the desired $u_{n+2}\textnormal{-}\Pi^1_3$ definition. 
By Martin \cite{martin_largest}, if $(y,r) \in \mathbb{R}^2$, $C\subseteq \mathbb{R}$ is  $\omega n \textnormal{-}\Pi^1_1(y,r)$, then there is a  formula $\varphi$ such that 
 {Player I has a winning strategy in $G(C)$}
  {iff } 
  \begin{displaymath}
L[y,r]\models \varphi(y,r,u_1,\ldots,u_n).
\end{displaymath}
The transform from the $\omega n\textnormal{-} \Pi^1_1(y,r)$ definition of $C$ to $\varphi$ is uniform, independent of $(y,r)$. 
Suppose $A = \game B$, where $B\subseteq \mathbb{R}^2$ is $\game(\omega n \textnormal{-}\Pi^1_1)$. Suppose $\varphi$ is a formula such that
\begin{displaymath}
  (y,r) \in B \eqiv~    L[y,r] \models \varphi(y,r,u_1,\ldots,u_n).
\end{displaymath}
To establish a $u_{n+2}\textnormal{-}\Pi^1_3$ definition of $A$, we have to decide which player has a winning strategy in $G(B_y)$, for $y \in \mathbb{R}$.
For ordinals  $\xi_1<\cdots<\xi_n<\eta< \omega_1$, we say that $M$ is a Kechris-Woodin non-determined set with respect to $(y,\xi_1,\ldots,\xi_n,\eta)$ iff
\begin{enumerate}
\item $M$ is a countable subset of $\mathbb{R}$;
\item $M$ is closed under join and Turing reducibility;
\item $\forall \sigma \in M ~ \exists v \in M ~ L_{\eta}[y,\sigma\otimes v] \models \neg\varphi(y,\sigma\otimes  v,\xi_1,\ldots,\xi_n)$;
\item $\forall \sigma \in M ~ \exists v \in M ~ L_{\eta}[y,v\otimes \sigma]  \models \varphi(y, v\otimes \sigma,\xi_1,\ldots,\xi_n)$.
\end{enumerate}
In clause 3, ``$\forall \sigma \in M $'' is quantifying over all strategies $\sigma$ for Player I that is coded in some member of $M$;  $\sigma*v$ is Player I's  response to $v$ according to $\sigma$, and $\sigma\otimes v = (\sigma*v)\oplus v$ is the combined infinite run. Similarly for clause 4, roles between two players being exchanged. Say that $z$ is $(y,\xi_1,\ldots,x_n,\eta)$-stable iff $z$ is not contained in any Kechris-Woodin non-determined set with respect to $(y,\xi_1,\ldots,\xi_n,\eta)$.  $z$ is $y$-stable iff $z$ is $(y,\xi_1,\ldots,\xi_n,\eta)$-stable for all $\xi_1<\ldots<\xi_n<\eta<\omega_1$. The set of $(y,z)$ such that $z$ is $y$-stable is $\Pi^1_2$. 
By the proof of Kechris-Woodin \cite{KW_det_transfer}, for all $y\in \mathbb{R}$, there is $z \in \mathbb{R}$ which is $y$-stable.

Note that if $z$ is $(y,\xi_1,\ldots,\xi_n,\eta)$-stable and $z \leq_T z'$, then $z'$ is $(y,\xi_1,\ldots,\xi_n,\eta)$-stable.
Let $<^{\xi_1,\ldots,\xi_n,\eta}_{y}$ be the following wellfounded relation on the set of $z$ which is $(y,\xi_1,\ldots,\xi_n,\eta)$-stable:
\begin{align*}
  z' <^{\xi_1,\ldots,\xi_n,\eta} _{y} z \eqiv~&  z \text{ is $(y,\xi_1,\ldots,\xi_n,\eta)$-stable}  \wedge z \leq_T z' \wedge \\
&  \forall \sigma \leq_T z~ \exists v \leq_T z' ~L_{\eta}[y,\sigma\otimes v] \models \neg\varphi(y,\sigma\otimes  v,\xi_1,\ldots,\xi_n)\\
&  \forall \sigma \leq_T z~ \exists v \leq_T z' ~ L_{\eta}[y,v\otimes \sigma] \models \varphi(y,v\otimes \sigma,\xi_1,\ldots,\xi_n).
\end{align*}
 Wellfoundedness of $<^{\xi_1,\ldots,\xi_n,\eta}_{y}$ follows from the definition of $(y,\xi_1,\ldots,\xi_n,\eta)$-stableness.  If $z $ is $(y,\xi_1,\ldots,\xi_n,\eta)$-stable, then $<^{\xi_1,\ldots,\xi_n,\eta}_{y} \res \set{z'}{z'<^{\xi_1,\ldots,\xi_n,\eta} _{y}z}$ is a $\boldsigma{1}$ wellfounded relation in parameters $(y,z)$ and  the code of $(\xi_1,\ldots,\xi_n,\eta)$, hence has rank $<\omega_1$ by Kunen-Martin.  If $z$ is $y$-stable,  let $f^z_y$ be the function that sends $(\xi_1,\ldots,\xi_n,\eta)$ to the rank of $z$ in $<^{\xi_1,\ldots,\xi_n,\eta}_{y}$. Then $f^z_y$ is a function into $\omega_1$. By $\Sigma^1_2$-absoluteness between $V$ and $L[y,z]^{\coll(\omega,\eta)}$, we can see $f^z_y \in L[y,z]$. Furthermore, $f^z_y$ is definable over $L[y,z]$ in a uniform way, so there is a $\se{\underline{\in}}$-Skolem term $\tau$ such that for all $(y,z) \in \mathbb{R}^2$, if $z$ is $y$-stable, then
 \begin{displaymath}
f_y^z(\xi_1,\ldots,\xi_n,\eta) = \tau^{L[y,z]}(y,z,\xi_1,\ldots,\xi_n,\eta).
\end{displaymath}
Let
\begin{displaymath}
\beta_y^z =  \tau^{L[y,z]}(y,z,u_1,\ldots,u_{n+1}).
\end{displaymath}
The function
\begin{displaymath}
  (y,z) \mapsto \beta^z_y
\end{displaymath}
 is $\Delta^1_3$ in the sharp codes.
We say that $z$ is  $y$-ultrastable iff $z$ is $y$-stable and   $\beta_y^z = \min\set{\beta_y^w}{w  \text{ is $y$-stable}}$.

\begin{claim}\label{claim:leftmost-branch-winning-strategy}
  If $z$ is $y$-ultrastable, then there is $\sigma \leq_T z$ such that $\sigma$ is a winning strategy for either of the players in $G(B_{y})$.
\end{claim}
\begin{proof}[Proof of Claim~\ref{claim:leftmost-branch-winning-strategy}]
  Suppose otherwise. For any $\sigma \leq_T z$ which is a strategy for either player, pick $w_{\sigma}$ which defeats $\sigma$ in $G(B_y)$. Let $w$ be a real coding $\set{(\sigma,w_{\sigma})}{\sigma \leq_T z}$. By an indiscernability argument, for any $(y,w)$-indiscernibles $\xi_1 <\cdots < \xi_n<\eta$, for any $\sigma \leq_T z$, if $\sigma$ is a strategy for Player I, then
  \begin{displaymath}
    L_{\eta}[y,\sigma \otimes w_{\sigma}] \models \neg\varphi(y,\sigma\otimes w_{\sigma}, \xi_1,\ldots,\xi_n);
  \end{displaymath}
if $\sigma$ is a strategy for Player II, then
  \begin{displaymath}
    L_{\eta}[y,w_{\sigma} \otimes \sigma]  \models \varphi(y,w_{\sigma}\otimes\sigma, \xi_1,\ldots,\xi_n).
  \end{displaymath}
This exactly means
\begin{displaymath}
  w<^{\xi_1,\ldots,\xi_n,\eta} _y z, 
\end{displaymath}
and hence
\begin{displaymath}
  f^w_y(\xi_1,\ldots,\xi_n,\eta) <  f^z_y(\xi_1,\ldots,\xi_n,\eta).
\end{displaymath}
Since $z$ is $y$-stable and $z \leq_T w$, $w$ is $y$-stable. Therefore, $\beta_y^{w}$ is defined and $\beta_y^{w}<\beta_y^z$, contradicting to  $y$-ultrastableness of $z$.
\end{proof}
From Claim~\ref{claim:leftmost-branch-winning-strategy}, if Player I (or II) has a winning strategy in $G(B_{y})$, then for any $y$-ultrastable $z$, there is a winning strategy for Player I (or II) in $G(B_{y})$ which is Turing reducible to $z$. Therefore, Player I has a winning strategy in $G(B_{y})$ iff there is   $\delta<u_{n+2}$ such that
\begin{align}
 \exists  z ~ (& z \text{ is $y$-stable}\wedge \beta_y^z   =\delta)
\label{eq:1}\\
\intertext{and}
   \forall \gamma \leq \delta~ \forall z ~ (& (z \text{ is $y$-stable} \wedge \beta_y^z = \gamma)\to \nonumber \\
&  \exists \sigma \leq_T z   ~(\sigma \text{ is a winning strategy for I in } G(B_{y}))).\label{eq:2}
\end{align}
Note that in~(\ref{eq:1}),
\begin{displaymath}
  \set{(\delta,y)}{ \exists  z  ~ (z \text{ is $y$-stable} \wedge \beta_y^z   =\delta )
}
\end{displaymath}
is a $\Sigma^1_3$ subset of $u_{\omega} \times \mathbb{R}$, and in~(\ref{eq:2}),
\begin{align*}
  \{ (\gamma,y) : \forall z ~ (& (z \text{ is $y$-stable} \wedge \beta_y^z = \gamma)\to \nonumber \\
&  \exists \sigma \leq_T z   ~(\sigma \text{ is a winning strategy for I in } G(B_{y})))\} 
\end{align*}
is a $\Pi^1_3$ subset of $u_{\omega} \times \mathbb{R}$. So $\exists \delta<u_{n+2} (\eqref{eq:1}\wedge \eqref{eq:2})$ is a $u_{n+2}\textnormal{-}\Pi^1_3$ definition of $A$.
\end{proof}

Lemma~\ref{lem:diff-Pi-1-3-to-game} and Lemma~\ref{lem:game-to-diff-Pi-1-3} are concluded in  a simple equality between pointclasses.
\begin{theorem}
  \label{thm:pointclass_game_exchange}
  Assume $\boldDelta{2}$-determinacy. Then for $x \in \mathbb{R}$, 
  \begin{displaymath}
\game^2 ( <\! \omega^2 \text{-}\Pi^1_1(x)) = \:<\!u_{\omega} \text{-}\Pi^1_3(x).
\end{displaymath}
\end{theorem}

\begin{definition}\label{def:OT2x}
  \begin{displaymath}
    \mathcal{O}^{T_2,x} = \{ (\gcode{\varphi},\alpha) : \varphi \text{ is a $\Sigma_1$-formula}, \alpha<u_{\omega}, \admistwo{x} \models \varphi(T_2,x,\alpha)\}.
  \end{displaymath}
\end{definition}

$\mathcal{O}^{T_2,x}$ is the $u_{\omega}$-version of Kleene's $\mathcal{O}$ relative to $(T_2,x)$. It is called $\mathcal{P}_3^x$ in \cite{Kechris_Martin_II}.

\begin{definition}\label{def:2sharp}
\begin{displaymath}
  x^{2\#}_{n} = \{(\gcode{\varphi},\gcode{\psi}) :  \exists \alpha<u_n ( (\gcode{\varphi},\alpha) \notin \mathcal{O}^{T_2,x}\wedge \forall \eta<\alpha (\gcode{\psi},\eta) \in \mathcal{O}^{T_2,x} )\}.
\end{displaymath}
\begin{displaymath}
  x^{2\#} = \{(n,\gcode{\varphi},\gcode{\psi}) :  n<\omega \wedge (\gcode{\varphi},\gcode{\psi}) \in x^{2\#}_n\}.
\end{displaymath}
\end{definition}

$\mathcal{O}^{T_2,x}$ splits into $\omega$ many parts $(\mathcal{O}^{T_2,x} \cap (\omega \times u_n))_{n<\omega}$. Each part is squeezed into a real $x^{2\#}_n$ by applying the difference operator on its second coordinate. The join of $(x^{2\#}_n)_{n<\omega}$ is  $x^{2\#}$. 
In particular, $x^{2\#}_0$ is Turing equivalent to the good universal $\Pi^1_3$ real, which is called the $\Delta^1_3$-jump of $x$. Each $x^{2\#}_n$ belongs to $\admistwo{x}$, but $x^{2\#} \notin \admistwo{x}$. The distinction between $x^{2\#}_0$ and $x^{2\#}$ does not have a lower level analog. 

The expression of $0^{2\#}$ generalizes Kleene's $\mathcal{O}$ to the higher level. 
Note that the transformations between $\game^2(\lessthanshort{\omega^2}\textnormal{-}\Pi^1_1(x))$ and $\lessthanshort{u_{\omega}}\textnormal{-}\Pi^1_3(x)$ definitions in Theorem~\ref{thm:pointclass_game_exchange} are uniform. Applying Theorem~\ref{thm:pointclass_game_exchange}  to the space $\mathcal{X} = \omega$, in combination with Theorem~\ref{thm:BK-KM}, we get the equivalence between $x^{2\#}$ and $M_1^{\#}(x)$.

\begin{theorem}\label{thm:2sharp-equivalent-to-M1sharp}
 Assume $\boldDelta{2}$-determinacy.  Then $x^{2\#}$ is many-one equivalent to $M_1^{\#}(x)$, the many-one reductions being independent of $x$.
\end{theorem}

$0^{2\#}$ is essentially a fancy way of expressing $y_3$, the leftmost real of $\widehat{T_2}$ which is used in the standard uniformization argument.  
$T_2$ and $y_{3}$ are used in  \cite{martin_solovay_basis_1969} to show that every nonempty $\Sigma^1_3$ set of reals contains a member which is recursive in $y_{3}$, or in our terminology, recursive in $0^{2\#}$. 
Basis theorems can also be proved with inner model theory. If $M_{2n-1}^{\#}$ exists, then every nonempty $\Sigma^1_{2n+1}$ set of reals contains a member recursive in $M_{2n-1}^{\#}$ (cf.  \cite{steel_projective_wo_1995,steel-handbook}). At higher levels, the leftmost real basis arguments are investigated in \cite{Q_theory}. It is shown by Harrington (modulo Neeman \cite{nee_opt_I,nee_opt_II}) that under $\boldDelta{2n}$-determinacy, there is a $\Delta^1_{2n+1}$-scale on a $\Delta^1_{2n+1}$ set whose leftmost real $y_{2n+1}$ is $\Delta^1_{2n+1}$-equivalent to $M_{2n-1}^{\#}$ and such that every nonempty  $\Sigma^1_{2n+1}$ set contains a real recursive in $y_{2n+1}$.  It is asked in \cite[Conjecture 11.2]{Q_theory} whether $y_{2n+1}$ is Turing equivalent to $M_{2n-1}^{\#}$. Theorem~\ref{thm:2sharp-equivalent-to-M1sharp} solves this conjecture in the $n=1$ case in an effective manner.

\section*{Acknowledgements}
The breakthrough ideas of this series of papers were obtained during the AIM workshop on Descriptive inner model theory, held in Palo Alto, and the Conference on Descriptive Inner Model Theory, held in Berkeley, both in June, 2014. 
The author greatly benefited from conversations with Rachid Atmai and Steve Jackson that took place in these two conferences. 
The final phase of this paper was completed whilst the author was a visiting fellow at the Isaac Newton Institute for Mathematical Sciences in the programme `Mathematical, Foundational and Computational Aspects of the Higher Infinite' (HIF) in August and September, 2015 funded by NSF Career grant DMS-1352034 and EPSRC grant EP/K032208/1.

\bibliography{sharp}{}
\bibliographystyle{plain}

\end{document}